\documentclass[11pt, reqno]{amsart}
\usepackage{amsmath,amsthm,amssymb,mathrsfs}

\usepackage[abbrev,non-sorted-cites]{amsrefs}
\usepackage{hyperref}

\usepackage{color,graphicx}
\usepackage{verbatim}
\usepackage{datetime}

\theoremstyle{plain}
  \newtheorem{thm}{Theorem}[section]
  \newtheorem{lem}[thm]{Lemma}
  \newtheorem{prop}[thm]{Proposition}
  \newtheorem{cor}[thm]{Corollary} 
\theoremstyle{definition}
  \newtheorem{defn}[thm]{Definition}
	
  \newtheorem{rmk}[thm]{Remark}
  \newtheorem*{ack*}{Acknowledgements}
  \newtheorem*{ques*}{Question}
\theoremstyle{plain}

\numberwithin{equation}{section}

\newcommand\ip[2]{\langle{#1},{#2}\rangle}

\newcommand\pl{\partial}

\newcommand\w{\wedge}

\newcommand\dt{\delta}

\newcommand\vep{\varepsilon}
\newcommand\vph{\varphi}
\newcommand\om{\omega}
\newcommand\ta{\theta}

\newcommand\af{\alpha}
\newcommand\bt{\beta}

\newcommand\ld{\lambda}

\newcommand\Om{\Omega}
\newcommand\Sm{\Sigma}

\newcommand\Ld{\Lambda}
\newcommand\Ta{\Theta}
\newcommand\Dt{\Delta}

\newcommand\CI{\mathcal{I}}

\newcommand\RO{\mathrm{O}}
\newcommand\RSO{\mathrm{SO}}

\newcommand\RGr{\mathrm{Gr}}

\newcommand\BR{\mathbb{R}}

\newcommand\ol{\overline}

\newcommand\ot{\otimes}
\newcommand\op{\oplus}

\newcommand\bx{\mathbf{x}}

\newcommand\bfI{\mathbf{I}}

\newcommand\dd{{\mathrm d}}

\newcommand\sff{\mathsf{f}}
\newcommand\rankdf{\mathsf{r}}
\newcommand\matrixG{{G}}
\newcommand\matrixH{{H}}

\DeclareMathOperator{\tr}{tr}

\DeclareMathOperator{\rank}{rank}

\DeclareMathOperator{\spn}{span}
\DeclareMathOperator{\sgn}{sign}
\DeclareMathOperator{\spt}{spt}

\DeclareMathOperator{\esssup}{ess\,sup}

\begin{document}

\title{Calibrating Forms for Minimal Graphs\\in Arbitrary Codimension}

\author{Chung-Jun Tsai}
\address{Department of Mathematics, National Taiwan University, and National Center for Theoretical Sciences, Math Division, Taipei 10617, Taiwan}
\email{cjtsai@ntu.edu.tw}

\author{Mu-Tao Wang}
\address{Department of Mathematics, Columbia University, New York, NY 10027, USA}
\email{mtwang@math.columbia.edu}



\thanks{C.-J.~Tsai is supported in part by the National Science and Technology Council grant 112-2628-M-002-004-MY4.  M.-T.~Wang is supported in part by the National Science Foundation under Grants DMS-2104212 and DMS-2404945, and by the Simons Foundation through the Simons Fellowship SFI-MPS-SFM-00006056.  This material is based upon work supported by the National Science Foundation under Grant No.~DMS-1928930, while M.-T.~Wang was in residence at the Simons Laufer Mathematical Sciences Institute (formerly MSRI) in Berkeley, California, during the Fall 2024 semester and C.-J.~Tsai visited during this period. Part of this work was carried out when M.-T.~Wang was visiting the Institute of Mathematics, Academia Sinica.
}

\begin{abstract}
We introduce a new family of closed differential forms naturally associated with minimal graphical submanifolds in Euclidean space, defined in arbitrary codimension. For each minimal graph, we construct an explicit closed form whose restriction coincides with the induced volume form. These forms admit a geometric interpretation as pullbacks, via the Gauss map, of tautological differential forms on the Grassmannian. In contrast to most known calibrations, they are generally not parallel and do not arise from special holonomy or symmetry considerations. The calibration problem is thus reduced to estimating the pointwise comass of the constructed forms. We show that the comass bound can be characterized in terms of explicit inequalities involving the singular values of the defining map of the graph, formulated via its two-dilations and we identify precise conditions ensuring that the comass is at most one. As a consequence, any minimal graph satisfying these conditions is calibrated and hence area-minimizing. This yields a broad class of new calibrated minimal graphs, extending the classical codimension-one theory, and provides an effective criterion for determining precisely where a given minimal graph is area-minimizing. As an application of our construction, we confirm a conjecture of Lawson and Osserman under two-dilation conditions, in arbitrary codimesnion. 
\end{abstract}

\maketitle


\section{Introduction}

A central problem in geometric analysis is to determine when a minimal submanifold, which is stationary for the area functional, is in fact \emph{area-minimizing} in its homology class.  Among the available tools, the theory of calibrations provides a particularly powerful and conceptually simple mechanism: the existence of a closed differential form whose pointwise comass is bounded by one and which restricts to the induced volume form on the submanifold immediately implies the area-minimizing property. Since the foundational work of Harvey and Lawson \cite{Harvey-Lawson-82}, calibrated geometry has played a key role in the study of minimal submanifolds, geometric measure theory, and nonlinear elliptic PDE.

Most classical examples of calibrations arise from \emph{parallel differential forms}, typically associated with special holonomy structures \cite{Harvey-Lawson-82}, or from highly symmetric constructions such as Lawlor’s angle criterion for cones \cites{Lawlor-89, Lawlor-91}. While these examples are geometrically rich, they apply only to relatively rigid classes of minimal submanifolds. In contrast, \emph{minimal graphs} in Euclidean space form a vast and flexible family, governed by nonlinear elliptic systems, yet comparatively few general calibration constructions are available for them. Even in Euclidean space, minimality alone does not imply that the submanifold is area-minimizing, and explicit calibrations for minimal graphs are known primarily in codimension one or in special integrable settings such as special Lagrangian geometry.

In codimension one, a well-known construction (see for example \cite{CM-11}*{Chapter 1}) associates to a minimal hypersurface a closed \((n-1)\)-form obtained from a divergence-free vector field that coincides with the unit normal along the hypersurface. In this setting, the area-minimizing property reduces to a pointwise bound on the magnitude of the vector field, which can be verified under explicit geometric conditions on the graph. Beyond codimension one, however, no comparable general mechanism is known.

The purpose of this paper is to introduce a \emph{new family of closed differential forms naturally associated with minimal graphical submanifolds in Euclidean space}, valid in arbitrary codimension. Given a minimal graph, we construct an explicit differential form that is closed and whose restriction to the graph coincides with the induced volume form. These forms may be interpreted as the pull-back, via the Gauss map, of certain tautological differential forms on the Grassmannian. In general, they are \emph{not parallel} and do not arise from special holonomy or symmetry considerations. Therefore, the calibration problem reduces entirely to estimating the pointwise comass.

A key feature of our approach is that the comass bound can be characterized in terms of explicit inequalities on the singular values of the defining map of the graph. We identify precise conditions on these singular values that ensure the comass of the constructed form is less than or equal to one. As a consequence, any minimal graph satisfying these inequalities is calibrated by our form and therefore area-minimizing. This yields a broad class of new calibrated minimal graphs, encompassing and extending the classical codimension-one construction.


We first recall the definitions of calibrations and calibrated submanifolds. 

\begin{defn}
    A differential form $\Theta$ on Euclidean space is called a calibration/calibrating form if it is closed, $\dd\Theta=0$, and has comass at most one. 
\end{defn}

The definition of comass will be recalled in Section~\ref{sec:comass}.

\begin{defn}
    An oriented submanifold $\Sigma$ of Euclidean space is said to be \emph{calibrated} (by \(\Theta\)) if there exists a calibration \(\Theta\) such that \[ \Theta|_{\Sigma} = \operatorname{vol}_{\Sigma}~,\] where $\operatorname{vol}_{\Sigma}$ denotes the volume form induced by the orientation.
\end{defn}

A calibrated submanifold is area-minimizing in its homology class \cite{Harvey-Lawson-82}.  The primary objects of interest in this work are graphical minimal submanifolds/minimal graphs,  defined below. 

\begin{defn}
Let $n, m \geq 1$ be integers. Let $\Omega\subset\mathbb{R}^{n}$ be a domain and $
F:\Omega\to\mathbb{R}^m$
be a smooth map. The graph of $F$ is the subset
\[
\Sigma \;=\;\bigl\{(x,F(x)) : x\in\Omega \bigr\}
\;\subset\;\mathbb{R}^{n}\times\mathbb{R}^m=\mathbb{R}^{n+m}.
\]
We say that $\Sigma$ is a \emph{minimal graph of dimension $n$} and \emph{codimension $m$}
if its mean curvature vanishes identically. Equivalently, $\Sigma$ is an $n$-dimensional minimal submanifold of $\mathbb{R}^{n+m}$.
\end{defn}

We first state our codimension-two theorem below. 
\begin{thm}
Let $F:\Omega\subset \mathbb{R}^n\rightarrow \mathbb{R}^2$ be a smooth map whose
graph is a minimal submanifold of $\mathbb{R}^{n+2}$. Suppose that $F$ is
area-non-increasing, i.e.\ its \(2\)-dilation is no greater than one:
\begin{equation}\label{area-dec-2}\lambda_1\lambda_2\leq 1 \end{equation} where $\{\lambda_i\}_{i=1,2}$ are the singular values of $\dd F$. Then the graph of $F$ is calibrated by $\Theta(F)$, defined in Definition~\ref{defn:pre-calibration}
\end{thm}

In fact, let $x_1,\ldots, x_n$ be the coordinates on $\mathbb{R}^n$, and $y_1, y_2$ be the coordinates on $\mathbb{R}^2$.  Write $F = (f, g)$, where $f, g:\Om\to\BR$. Then, the condition \eqref{area-dec-2} in terms of $f, g$ is 
\begin{equation}|\nabla f|^2|\nabla g|^2-(\nabla f\cdot \nabla g)^2\leq 1\end{equation} and the graph of $F$ is calibrated by the $n$-form  
\begin{equation}\begin{split}
\displaystyle
\Theta(F)=\frac{1}{\sqrt{\Xi}}\Bigl(
&\left(1-|\nabla f|^2|\nabla g|^2+(\nabla f\cdot \nabla g)^2\right)\dd x_1\wedge\cdots \wedge\dd x_n\\
& + \dd y_1\wedge \left((1+|\nabla g|^2)*\dd f - (\nabla f\cdot \nabla g)*\dd g\right)\\
& + \dd y_2\wedge \left(-(\nabla f\cdot \nabla g)*\dd f + (1+|\nabla f|^2)*\dd g\right)
\Bigr).\end{split}
\end{equation}
where $\Xi=(1 + |\nabla f|^2)\,(1 + |\nabla g|^2) - (\nabla f\cdot\nabla g)^2$, and $*$ denotes the Hodge star on $\mathbb{R}^n$; thus $*\dd f$ and $*\dd g$ are $(n-1)$-forms on $\mathbb{R}^n$.

The connection between such an area-decreasing (or non-increasing) condition and minimal submanifold of higher codimension was made in \cite{Wang-03} and can be understood in terms of the Gauss map and Grassmannian geometry. In particular, it is known that such conditions imply stability and uniqueness of minimal graphs \cites{Lee-Tsui-14, Lee-Ooi-Tsui-19, Lee-Wang-03, Lee-Wang-08}.

Our theorem holds in any codimension under the following 2-dilation condition.

\begin{thm} (see Theorem~\ref{thm:2-dilation-crude})
Let $F:\Omega\subset \mathbb{R}^n\rightarrow \mathbb{R}^m$ be a smooth map whose
graph is a minimal submanifold of $\mathbb{R}^{n+m}$. Suppose that \( \rankdf = \sup_{\bx\in\Om}\bigl(\rank \dd F|_{\bx}\bigr) \geq 2 \). If the singular values $\{\lambda_i\}$ of \(\dd F\) satisfy
    \begin{align} \label{eqn:2-dilation}
        \ld_j\ld_k \leq \frac{1}{(\rankdf-1)^2}
    \end{align}
    for any \(1\leq j<k\leq \rankdf\) and at every point of \(\Om\), then the graph of $F$ is a calibrated submanifold of $\mathbb{R}^{n+m}$.
\end{thm}
Again, $F$ is calibrated by $\Theta(F)$ defined in Definition~\ref{defn:pre-calibration}. We remark that when $\rankdf\leq 1$, our theorem recovers the codimension-one case.

Since any smooth minimal submanifold of Euclidean space is locally graphical, and condition \eqref{eqn:2-dilation} is automatically satisfied in a neighborhood of each point, our result implies that every smooth minimal submanifold in Euclidean space is locally area-minimizing, see Bryant \cite{Bryant}\footnote{Our construction is nevertheless fundamentally different from that proposed in \cites{Bryant, Lawlor-89, Lawlor-91}, where the approach relies on a volume-decreasing map onto the submanifold $\Sigma$, followed by pulling back the volume form of $\Sigma$ via this map. In contrast, our construction begins with a tautological form on $\Sigma$ arising from the graphical assumption, and then extends this form by parallel transport along the $\mathbb{R}^m$ direction.}. In particular, this yields an alternative proof of the theorem of Lawlor and Morgan \cite{Lawlor-Morgan-96}, who obtained the same result using curvy slicing. Moreover, it provides an effective criterion for determining exactly where a given minimal graph is area-minimizing.

The upper bound on the right-hand side of \eqref{eqn:2-dilation} can be replaced by a bound of the order $\frac{1}{\rankdf-1}$, which is sharper when $\rankdf$ is larger, see Theorem~\ref{thm:2-dilation-refined}.

As an application of this construction, we confirm a conjecture \cite{Lawson-Osserman-77}*{Conjecture 2.1} of Lawson and Osserman concerning Lipschitz solutions of the minimal graph system under 2-dilation assumptions\footnote{Lawson and Osserman referred to it as the  ``minimal surface system", we use ``minimal graph system" to emphasize that it applies in arbitrary dimension and codimension.}.

\begin{defn}
    A smooth map $(f_1,\cdots,f_m):\Omega\subset \mathbb{R}^n\rightarrow \mathbb{R}^m$ is said to satisfy the \emph{minimal graph system} if 
    \begin{align} \label{eqn:minimal-graph}
    \sum_{j,k=1}^n \pl_j\bigl( \sqrt{g}(g^{-1})^{jk} (\pl_k f_\af) \bigr) &= 0
    \quad\text{for }\af = 1,\ldots,m ~,\end{align} where \begin{align}
    g_{jk} = \dt_{jk} + \sum_{\af=1}^m(\pl_jf_\af)(\pl_kf_\af) \quad, \quad g = \det(g_{jk}),
\end{align} and $(g^{-1})^{jk}$ is the inverse of $g_{jk}$.
    
\end{defn}

 For a smooth map, the minimal graph system is the Euler--Lagrange equation of the volume functional, and is equivalent to the stationarity of the graph under \emph{ambient} variations.  Their conjecture asserts that this equivalence should remain valid for Lipschitz graphs: that is, every Lipschitz weak solution of the minimal graph system should already satisfy the full geometric Euler–Lagrange condition of graph stationarity.  If this conjecture holds, one can connect the minimal graph system to the machinery of geometric measure theory, including monotonicity, blow-up analysis, etc.  Recently, Dimler \cite{Dimler-23}*{Theorem 5.4} proved a partial form of the conjecture under an additional invariance assumption; Hirsch, Mooney, and Tione \cite{Hirsch-Mooney-Tione-23} proved the conjecture when the domain is $2$-dimensional. We confirm the conjecture under \(2\)-dilation conditions for any dimension $n$ and codimension $m$.

\begin{thm} (see Corollary~\ref{LO}) The Lawson--Osserman conjecture holds for any Lipschitz weak solution $F:\Omega\subset \mathbb{R}^n \rightarrow \mathbb{R}^m$ of the minimal graph system that satisfies the condition \eqref{eqn:2-dilation} almost everywhere. Moreover, such an \(F\) is smooth.
\end{thm}

The paper is organized as follows. In Section~\ref{sec:form}, we define the form $\Theta(F)$ for any smooth map $F:\Omega \subset \mathbb{R}^n\rightarrow \mathbb{R}^m$. In Section~\ref{sec:closed} we present a coordinate expression for $\Theta(F)$ and show that it is closed if and only if $F$ is a minimal map.   Section~\ref{sec:comass} is devoted to the comass estimates and the geometric conditions that ensure calibration. Finally, in Section~\ref{sec:Lipschitz} we discuss Lipschitz solutions of the minimal graph system and confirm the Lawson--Osserman conjecture under 2-dilation conditions.

\begin{ack*}
    The first-named author thanks Ulrich Menne and Chen-Kuan Lee for helpful discussions on geometric measure theory.
\end{ack*}

\section{Differential Forms Associated with a Graph} \label{sec:form}

Let \(\Om\subset\BR^n\) be a domain, and \(F:\Om\to\BR^m\) be a smooth map.  In this section, we construct an \(n\)-form $\Theta(F)$ on \(\Om\times\BR^m\) associated with the map \(F\).

We present three different, yet equivalent, definitions of $\Theta(F)$. The first (Definition \ref{defn:pre-calibration}) is coordinate- and frame-independent\footnote{The definition depends on the choice of the base $\mathbb{R}^n$, however.} and can be interpreted as the pullback, via the Gauss map, of tautological forms on the Grassmannian. The second \eqref{form:Theta-tangent-normal} relies on the singular value decomposition of the differential of the defining map $F$ and is well suited for estimating the comass. The third (\eqref{form:Theta-g} and \eqref{form:Theta-h-0}) is expressed in terms of the components of the defining map $F$, allowing us to connect with the minimal surface system and to prove the closedness of $\Theta(F)$.

The graph of \(F\), \(\Sm = \{(x,F(x)):x\in\Om\}\), is a submanifold in \(\Om\times\BR^m\).  The key step is to construct an \(n\)-form on \(\Sm\), and then extend it by parallel transport along the \(\BR^m\)-factor. 

\begin{defn} \label{defn:tangent-to-normal-map}
We first introduce a map \(J:T\Sm\to N\Sm\) from the tangent bundle of $\Sigma$ to the normal bundle of $\Sigma$.
Fix a point \(p\in\Sm\) and write $\pi_1(p)\in \Omega$ for its projection onto $\mathbb{R}^n$. Define matrices 
    \begin{align*}
        \matrixG_p = \bfI_n + (\dd F|_{\pi_1(p)})^*(\dd F|_{\pi_1(p)})
        \qquad \text{and} \qquad
        \matrixH_p = \bfI_m + (\dd F|_{\pi_1(p)})(\dd F|_{\pi_1(p)})^* ~,
    \end{align*}
    where \(\dd F|_{\pi_1(p)}: T_{\pi_1(p)}\Om\cong\BR^n \to T_{F(\pi_1(p))}\BR^m\cong\BR^m\), and \(*\) denotes the adjoint matrix.
    Define an isometric isomorphism $L_p: \BR^n \rightarrow T_p\Sm$ by
    \begin{align*}
        L_p(v) = \bigl( \matrixG_p^{-\frac{1}{2}}(v), (\dd F|_{\pi_1(p)})\circ\matrixG_p^{-\frac{1}{2}}(v)\bigr) ~,
    \end{align*}
    and an isometric isomorphism  $L^\perp_p:\BR^m\rightarrow N_p\Sm$ by
    \begin{align*}
        L^\perp_p(w) = \bigl( -(\dd F|_{\pi_1(p)})\circ\matrixH_p^{-\frac{1}{2}}(w), \matrixH_p^{-\frac{1}{2}}(w)\bigr) ~.
    \end{align*}
    Finally, \(J_p:T_p\Sm\to N_p\Sm\) is defined to be
    \begin{align} \label{map:tangent-to-normal}
        J_p &= L_p^\perp\circ(\dd F|_{\pi_1(p)})\circ(L_p)^{-1}
    \end{align}

    With the help of Lemma~\ref{lem:linear-alg} below, one can verify that \(J_p\) can also be expressed as:
    \begin{align} \label{map:tangent-to-normal-nonsym} \begin{split}
        J_p &= \pi^{N_p\Sm} \circ \iota_2 \circ\,(\dd F|_{\pi_1(p)}) \circ \matrixG_p \circ (\pi_1|_{T_p\Sm}) \\
        &= \pi^{N_p\Sm} \circ \iota_2 \circ \matrixH_p \circ\,(\dd F|_{\pi_1(p)}) \circ (\pi_1|_{T_p\Sm})
    \end{split} \end{align}
    where \(\pi_1\) is the projection from \(\BR^n\times\BR^m\) onto \(\BR^n\), \(\iota_2\) the inclusion from \(\BR^m\) into \(\{0\}\times\BR^m\subset\BR^n\times\BR^m\), and \(\pi^{N_p\Sm}\) is the orthogonal projection from \(\BR^n\op\BR^m\cong T_p\Sm\op N_p\Sm\) onto \(N_p\Sm\).
\end{defn}

By contracting with \(J\) and using the metric dual, we construct an operator on differential forms and a family of differential forms $\Psi^\ell(F), \ell=0, 1, 2, 3\ldots$ associated to any smooth map $F:\Omega\subset \mathbb{R}^n\rightarrow \mathbb{R}^m$.
\begin{defn} \label{defn:component-pre-calibration}
    Let \(\Psi\) be the endomorphism of \(C^\infty(\Sm,\Ld^k(\BR^n\times\BR^m))\) defined by
    \begin{align*}
        \Psi(\xi) &= \sum_{i=1}^n (J(e_i))^\flat \w \iota_{e_i}\xi ~,
    \end{align*}
    where \(e_1,\ldots,e_n\) is any orthonormal basis of \(T\Sm\), $\iota_{e_i}$ denotes interior multiplication, and \(\flat:N\Sm \to (N\Sm)^*\) is the musical isomorphism of \(N\Sm\).  It is straightforward to verify that \(\Psi\) does not depend on the choice of basis.

    Fix an orientation of \(\BR^n\).  Since \(\Sm\) is graphical over \(\BR^n\), this fixes a volume form \(\operatorname{vol}_{\Sigma}\) on \(\Sm\).  For any non-negative integer \(\ell\), define the iteration:
    \begin{align}\label{eq:Psi}
        \Psi^\ell(F) &= (\Psi)^\ell(\operatorname{vol}_{\Sigma}) ~.
    \end{align}
    Note that \(\Psi^0(F) = \operatorname{vol}_{\Sigma}\).  At every \(p\in\Sm\), \(\Psi^\ell(F)\) lies in the canonical summand \( \Ld^{\ell}(N_p\Sm)^*\ot\Ld^{n-\ell}(T_p^*\Sm) \) of \(\Ld^n(\BR^n\times\BR^m)\). Parallel transporting along the \(\BR^m\) summand defines an \(n\)-form on \(\Om\times\BR^m\), which will still be denoted by \(\Psi^\ell(F)\).
\end{defn}

\begin{rmk}
    Since \(J\) depends only on the tangent plane, one can define \(J\) as a bundle homomorphism from the tautological bundle over the graph chart of \(\RGr(n,n+m)\) to the orthogonal tautological bundle over \(\RGr(n,n+m)\).  For a graphical submanifold, pulling it back via the Gauss map recovers the \(J\)-map defined in Definition~\ref{defn:tangent-to-normal-map}.  By using \(\RGr(n,n+m) = \RO(n+m)/(\RO(n)\times\RO(m))\), it is equivalent to define an \(\RO(n)\times\RO(m)\)-equivariant \(J\) on the graph chart of \(\RO(n+m)\).  Write \(W\in\RO(n+m)\) as
    \begin{align*}
        W &= \begin{bmatrix}
        W_1 & W_2 \\ W_3 & W_4
        \end{bmatrix}
        = \begin{bmatrix}
            \vec{w}_1 &\cdots &\vec{w}_n &\vec{w}_{n+1} &\cdots &\vec{w}_{n+m}
        \end{bmatrix} ~,
    \end{align*}
    where the first expression uses the \((n+m)\times(n+m)\) block matrices, and the second expression consists of column vectors.  The map \(\RO(n+m)\to\RGr(n,n+m)\) sends \(W\) to \(\spn\{\vec{w}_1,\cdots,\vec{w}_n\}\).  The graph chart is where \(W_1\) is invertible.  The map \(J\) at \(W\) is a homomorphism from \(\spn\{\vec{w}_1,\cdots,\vec{w}_n\}\) to \(\spn\{\vec{w}_{n+1},\cdots,\vec{w}_{n+m}\}\) defined by
    \begin{align*}
        J(\vec{w}_i) &= \sum_{\af=1}^m \bigl( W_4^*\,W_3\,(W_1)^{-1}\,(W_1^*)^{-1} \bigr)_{\af i} \, \vec{w}_{n+\af} ~.
    \end{align*}
    It is not hard to verify that \(J\) is \(\RO(n)\times\RO(m)\)-equivariant, and hence descends to \(\RGr(n,n+m)\).

    Similarly, \(\operatorname{vol}_{\Sigma}\) comes from the tautological volume form with \(\sgn(\det(W_1))\) on the graph chart of the Grassmannian, and all the \(\Psi^\ell(F)\) can be defined on the Grassmannian.
\end{rmk}

By using the singular value decomposition (SVD) of \(\dd F|_{\pi_1(p)}\), one obtains a concrete expression for \(\Psi^\ell(F)\).  There exists an \(\RO(n)\times\RSO(m)\) change of coordinates on \(\BR^n\times\BR^m\) so that
\begin{align} \label{eqn:SVD}
    \dd F|_{\pi_1(p)}\bigl(\frac{\pl}{\pl x_j}\bigr) &= \ld_j\frac{\pl}{\pl y_j} ~,
\end{align}
for all \(j\), where \(\ld_1\geq\cdots\geq\ld_n\geq0\) are the eigenvalues of \(\sqrt{(\dd F|_{\pi_1(p)})^*(\dd F|_{\pi_1(p)})}\).  When \(j > \rank(\dd F|_{\pi_1(p)})\), \(\ld_j = 0\).  If \(n > m\), introduce dummy variables \(y_j\) for \(j > m\).  With \eqref{eqn:SVD}, introduce the frame:
\begin{align} \label{eqn:SVD-frame}
    e_i = \frac{1}{\sqrt{1+\ld_i^2}}\bigl(\frac{\pl}{\pl x_i} + \ld_i\frac{\pl}{\pl y_i}\bigr)
    \quad\text{and}\quad
    e_{n+i} = \frac{1}{\sqrt{1+\ld_i^2}}\bigl(-\ld_i\frac{\pl}{\pl x_i} + \frac{\pl}{\pl y_i}\bigr)
\end{align}
for \(i=1,\ldots,n\).  Denote by \(\{\om^j\}_{1\leq j\leq 2n}\) the dual coframe.

It follows that \(\operatorname{vol}_{\Sigma}|_p = \om^1\w\cdots\w\om^n\), \(\{e_j\}_{j=1}^n\) is an oriented, orthonormal basis for \(T_p\Sm\), and \(\{e_{n+j}\}_{j=1}^m\) is an orthonormal basis for \(N_p\Sm\).  Moreover, for \(j=1,\ldots, n\),
\begin{align*}
    J_p(e_j) &= \ld_j e_{n+j} ~.
\end{align*}
We are now ready to describe \(\Psi^\ell(F)\) by using this frame.
Fix \(\ell\in\{1,\ldots,n\}\). Let $\CI \subset \{1, \cdots, n\}$ be a subset with \(|\CI| = \ell\).  Write \(\CI = \{i_1,\ldots,i_\ell\}\) where  \(1\leq i_1< \cdots< i_\ell\leq n\), and define a function $\chi_\CI:\{1,\cdots,n\} \rightarrow \{1, \cdots, 2n\}$ by
\begin{align} \label{fct:index}
    \chi_\CI(j) = \begin{cases}
        j &\text{if }j\notin\CI \\
        n+j &\text{if }j\in\CI
    \end{cases} ~,
\end{align}
Using this, define the $n$-form 
\begin{align} \label{form:special-basis}
    \Dt_\CI &= \om^{\chi_\CI(1)}\w\om^{\chi_\CI(2)}\w\cdots\w\om^{\chi_\CI(n)} ~.
\end{align}
In other words, in the wedge product defining $\Dt_\CI$, the \(j^{\text{th}}\) factor of \(\Dt_\CI\) is \(\om^j\) if \(j\notin\CI\), and is \(\om^{n+j}\) if \(j\in\CI\).  It is not hard to see that
\begin{align} \label{eqn:component-pre-calibration-SVD}
    \Psi^\ell(F) &= (\ell!)\, \sum_{\CI:|\CI|=\ell} (\prod_{j\in\CI}\ld_j)\Dt_\CI ~.
\end{align}

\subsection{The definition of $\Theta(F)$}

The $n$-form $\Theta(F)$ is a particular linear combination of \(\Psi^\ell(F)\):
\begin{defn} \label{defn:pre-calibration}
    For a domain \(\Om\subset\BR^n\) and a smooth map \(F:\Om\to\BR^m\), define an \(n\)-form on \(\Om\times\BR^m\) by
    \begin{align}
        \Ta(F) &= \Psi^0(F) - \sum_{\ell\geq 1}(-1)^\ell\frac{\ell - 1}{\ell!}\Psi^\ell(F) \\
        &= \Psi^0(F) - \frac{1}{2}\Psi^2(F) + \frac{2}{3!}\Psi^3(F) - \frac{3}{4!}\Psi^4(F) \pm \cdots ~. \notag
    \end{align}
    Since \(1 - \frac{1}{2}x^2 + \frac{2}{3!}x^3 - \frac{3}{4!}x^4 \pm \cdots = (1+x)\exp(-x)\), \(\Ta(F)\) may be formally written as
   \begin{align*}
        \Ta(F) &= \bigl((1+\Psi)\exp(-\Psi)\bigr)(\operatorname{vol}_{\Sigma}) = (\Psi^0(F)+\Psi^1(F))\exp(-\Psi^1(F)) ~.
    \end{align*}
\end{defn}

Clearly \(\Ta(F)|_\Sm = \operatorname{vol}_{\Sigma}\), so \(\Ta(F)\) is a natural candidate for a calibration of \(\Sm\).

In terms of the SVD frame \eqref{eqn:SVD-frame},
\begin{align} \label{form:Theta-tangent-normal}
    \Ta(F) &= \om^1\w\cdots\w\om^n - \sum_{\ell\geq2} (-1)^\ell(\ell-1) \bigl[ \sum_{\CI:|\CI| = \ell}(\prod_{j\in \CI}\ld_{j})\,\Dt_{\CI} \bigr] ~.
\end{align}
For example, when \(\rank(\dd F|_{\pi_1(p)})\leq 2\),
\begin{align} \label{form:Theta-rank-2}
    \Ta(F) &= (\om^1\w\om^2 - \ld_1\ld_2\,\om^{n+1}\w\om^{n+2})\w\om^3\w\cdots\w\om^n ~;
\end{align}
when \(\rank(\dd F|_{\pi_1(p)})= 3\),
\begin{align} \label{form:Theta-rank-3} \begin{split}
    \Ta(F) &= \Bigl( \om^1\w\om^2\w\om^3 - \ld_1\ld_2\,\om^{n+1}\w\om^{n+2}\w\om^3 - \ld_2\ld_3\,\om^1\w\om^{n+2}\w\om^{n+3} \\
    &\qquad - \ld_1\ld_3\,\om^{n+1}\w\om^2\w\om^{n+3} + 2\ld_1\ld_2\ld_3\,\om^{n+1}\w\om^{n+2}\w\om^{n+3} \Bigr)\w\om^4\w\cdots\w\om^n ~.
\end{split} \end{align}

\section{The Closedness Condition} \label{sec:closed}

The main purpose of this section is to investigate when \(\dd\Ta(F) = 0\).  

\subsection{A Coordinate Expression of $\Theta(F)$}

We first write \(\Ta(F)\) using the components of \(F\).

\begin{prop} \label{prop:calibration}
    For \(F = (f_1,\cdots,f_m):\Om\subset\BR^n\to\BR^m\), the \(n\)-form \(\Ta(F)\) on \(\Om\times\BR^m\) defined in Definition~\ref{defn:pre-calibration} is equal to
    \begin{align} \label{form:Theta-h-0}
        \sqrt{h}\bigl(\tr[(h^{-1})^{\af\bt}] - (m-1)\bigr)(*1) + \sum_{\af,\bt=1}^m \dd y_\af\w\sqrt{h}(h^{-1})^{\af\bt} (*\dd f_\bt)
    \end{align} 
    where \(*\) is the Hodge star on \(\BR^n\), 
    \[h_{\af\bt} =\dt_{\af\bt} + \sum_{j=1}^n(\pl_jf_\af)(\pl_jf_\bt),\quad h = \det(h_{\alpha\beta}),\] and $(h^{-1})^{\alpha\beta}$ is the inverse of $h_{\alpha\beta}$.
\end{prop}

\begin{proof}
The first step is to express \eqref{form:Theta-tangent-normal} in terms of the SVD coordinates, namely, \(\dd x_j\) and \(\dd y_j\) corresponding to \eqref{eqn:SVD}.  We now show that
\begin{align} \label{form:Theta-SVD}
    \Ta(F) &= \sqrt{\prod_{i=1}^n(1+\ld_i^2)} \Bigl( \bigl(1 - \sum_{j=1}^n\frac{\ld_j^2}{1+\ld_j^2}\bigr) \dd x_1\w\cdots\w\dd x_n + \sum_{k=1}^n (-1)^{k-1}\frac{\ld_k}{1+\ld_k^2}\dd y_k\w\widehat{\dd x_k} \Bigr) ~,
\end{align}
where \(\widehat{\dd x_k} = \dd x_1\w\cdots\w\dd x_{k-1}\w\dd x_{k+1}\w\cdots\w\dd x_n\).  To verify \eqref{form:Theta-SVD}, we plug the basis \eqref{eqn:SVD-frame} into the right-hand side of \eqref{form:Theta-SVD}.  We compute
\begin{align*}
    \Bigl[\text{right-hand side of \eqref{form:Theta-SVD}}\Bigr] (e_{1},\cdots,e_{n})
    &= 1 - \sum_{j=1}^n\frac{\ld_j^2}{1+\ld_j^2} + \sum_{k=1}^n \frac{\ld_k^2}{1+\ld_k^2} = 1 ~,
\end{align*}
and
\begin{align*}
    &\quad\Bigl[\text{right-hand side of \eqref{form:Theta-SVD}}\Bigr] (e_{n+1},e_{n+2},\cdots,e_{n+\ell},e_{\ell+1},e_{\ell+2},\cdots,e_{n}) \\
    &= \bigl(1 - \sum_{j=1}^n\frac{\ld_j^2}{1+\ld_j^2}\bigr)\prod_{k=1}^\ell(-\ld_k) + \sum_{j=1}^\ell\bigl(\frac{\ld_j}{1+\ld_j^2}\frac{\prod_{k=1}^\ell(-\ld_k)}{-\ld_j}\bigr) + \sum_{j=\ell+1}^n\bigl(\frac{\ld_j}{1+\ld_j^2}\ld_j{\prod_{k=1}^\ell(-\ld_k)}\bigr) \\
    &= (-1)^{\ell}\bigl(\prod_{k=1}^\ell\ld_k\bigr) \Bigl( 1 - \sum_{j=1}^\ell\frac{\ld_j^2}{1+\ld_j^2} - \sum_{j=1}^\ell\frac{1}{1+\ld_j^2} \Bigr) = (\ell-1)(-1)^{\ell-1}\ld_1\ld_2\cdots\ld_\ell ~.
\end{align*}
This verifies \eqref{form:Theta-SVD}.


Next utilizing the \(\RSO(n)\times\RO(m)\) invariance of \(\Ta(F)\), we derive its expression in terms of the components of \(F\).
Let \((f_1,\cdots,f_m)\) be the components of \(F:\Om\to\BR^m\).  In the SVD coordinate \eqref{eqn:SVD}, \(\dd f_k = \ld_k\dd x_k\).  Thus, \(*\dd f_k = (-1)^{k-1}\ld_k\widehat{\dd x_k}\), where \(*\) is the Hodge star with respect to the standard metric and orientation on \(\BR^n\).  Hence, the last term in \eqref{form:Theta-SVD} is the pairing between \(\dd y_k\) and \(*\dd f_k\) with the weight \(1/(1+\ld_k^2)\). Let
\begin{align} \label{eqn:normal-metric}
    h_{\af\bt} &= \dt_{\af\bt} + \ip{\dd f_\af}{\dd f_\bt} = \dt_{\af\bt} + \sum_{j=1}^n(\pl_jf_\af)(\pl_jf_\bt)
\end{align}
for \(1\leq\af,\bt\leq m\).  In the SVD coordinate \eqref{eqn:SVD}, \(h_{\af\bt} = (1+\ld_{\af}^2)\dt_{\af\bt}\).  It follows that
\begin{align} \label{form:Theta-h}
    \Ta(F) &= \sqrt{h}\bigl(\tr[(h^{-1})^{\af\bt}] - (m-1)\bigr)(\dd x_1\w\cdots\w\dd x_n) + \sum_{\af,\bt=1}^m \dd y_\af\w\sqrt{h}(h^{-1})^{\af\bt} (*\dd f_\bt) ~,
\end{align}
where \(h = \det(h_{\af\bt}) = \prod_k(1+\ld_k^2)\).  Note that the right-hand side of \eqref{form:Theta-h} is invariant under the \(\RSO(n)\times\RO(m)\) change of coordinates.  In other words, the \(x_i\) and \(y_\af\) coordinates on the right-hand side of \eqref{form:Theta-h} need not be the SVD coordinates \eqref{eqn:SVD}.
\end{proof}

\subsection{The closedness of $\Theta(F)$ and the minimal graph system}

 We recall Sylvester's determinant identity and some classical results in linear algebra.
\begin{lem} \label{lem:linear-alg}
    For an \(m\times n\) matrix \(S\), let \(g(S) = \bfI_n + S^*S\) and \(h(S) = \bfI_m + SS^*\).  Then,
    \begin{align*}
        \det(g(S)) &= \det(h(S)) ~, & h(S)^{-1} = \bfI_m - S\,g(S)^{-1}\,S^* ~, \\
        \tr(h(S)^{-1}) - m &= \tr(g(S)^{-1}) - n ~, & g(S)^{-1} = \bfI_n - S^*\,h(S)^{-1}\,S ~.
    \end{align*}
    Moreover, \(h(S)^{-1}S = Sg(S)^{-1}\).
\end{lem}

The proof of this lemma will be omitted.  We are ready to identify when $d\Theta(F)=0$.

\begin{thm} \label{thm:calibration}
    For \(F = (f_1,\cdots,f_m):\Om\subset\BR^n\to\BR^m\), the \(n\)-form \(\Ta(F)\) on \(\Om\times\BR^m\) defined in Definition~\ref{defn:pre-calibration} is closed if and only if the graph of \(F\) is a minimal submanifold.
\end{thm}

\begin{proof}

Since the coefficient function of \(\dd x_1\w\cdots\w\dd x_n\) in \eqref{form:Theta-h} depends only on \(x\), \(\dd\Ta(F) = -\sum_{\af,\bt=1}^m \dd y_\af\w\dd\bigl[\sqrt{h}(h^{-1})^{\af\bt} (*\dd f_\bt)\bigr]\).  We will relate it to the geometry of the graph \(\Sm\).

Let \(g_{jk}\) be the induced metric on \(\Sm\):
\begin{align}
    g_{jk} = \dt_{jk} + \sum_{\af=1}^m(\pl_jf_\af)(\pl_kf_\af) \quad\text{and}\quad g = \det(g_{jk}) ~.
\end{align}
It turns out that \eqref{form:Theta-h} can be rewritten using \(g_{jk}\). 
Applying Lemma~\ref{lem:linear-alg} with \(S = \dd F\), we find that
\[ \tr[(h^{-1})^{\af\bt}] - m = \tr[(g^{-1})^{jk}] - n \quad\text{and}\quad \sum_{\bt=1}^m (h^{-1})^{\af\bt}(\pl_jf_\bt) = \sum_{k=1}^n(g^{-1})^{jk}(\pl_kf_\af) ~. \]
Hence,
\begin{align*}
    \sum_{\bt=1}^m (h^{-1})^{\af\bt}(*\dd f_\bt) &= *\Bigl[ \sum_{k=1}^n(g^{-1})^{jk}(\pl_kf_\af) \dd x_j\Bigr]
    = \bigl(\sum_{j=1}^n\sqrt{g}(g^{-1})^{jk} (\pl_k f_\af)\frac{\pl}{\pl x_j}\bigr) \mathbin\lrcorner(*1)
\end{align*}
where \(*\) is the Hodge star with respect to the standard metric and orientation on \(\BR^n\).  It follows that
\begin{align} \label{form:Theta-g}
    \Ta(F) & = \sqrt{g}\bigl(\tr[(g^{-1})^{jk}] - (n-1)\bigr)(*1) + \sum_{\af=1}^m \dd y_\af\w\Bigl[ \bigl(\sum_{i,j=1}^n\sqrt{g}(g^{-1})^{jk} (\pl_k f_\af)\frac{\pl}{\pl x_j}\bigr) \mathbin\lrcorner(*1) \Bigr] ~.
\end{align}

Therefore, \(\dd\Ta(F) = 0\) if and only if
\begin{align} \label{eqn:minimal-standard}
    \sum_{j,k=1}^n \pl_j\bigl( \sqrt{g}(g^{-1})^{jk} (\pl_k f_\af) \bigr) &= 0
    \quad\text{for }\af = 1,\ldots,m ~,
\end{align}
which is exactly the \emph{minimal graph system} \eqref{eqn:minimal-graph} of \(F = (f_1,\cdots,f_m)\).
\end{proof}

\begin{rmk}
    In the hypersurface case, \(m=1\), \(\Ta(F)\) is the parallel transport of the volume form of its graph along the \(y\)-direction, and \eqref{form:Theta-g} reads
    \begin{align*}
        \Ta(f) &= \frac{1}{\sqrt{1+|\dd f|^2}}(*1) + \dd y \w\bigl(\frac{1}{\sqrt{1+|\dd f|^2}}*\dd f\bigr) ~.
    \end{align*}
\end{rmk}

\section{Estimating the Comass} \label{sec:comass}

Theorem~\ref{thm:calibration} characterizes when \(\dd\Ta(F) = 0\). The goal of this section is to identify conditions under which \(\Ta(F)\) has comass one. 

\begin{defn}
    The comass of a differential $n$-form $\Theta$ on Euclidean space is the supremum of the values of \(\Theta\) over all oriented unit \(n\)-planes.
\end{defn}

Therefore, estimating an upper bound of comass becomes an optimization problem over the Grassmannian.

Evaluating the comass is a pointwise computation, and we will use the expression \eqref{form:Theta-tangent-normal}.  Note that if \(\rank \dd F \leq 1\) everywhere, then \(\Ta(F)\) has comass equal to \(1\).

We adopt the SVD basis \eqref{eqn:SVD-frame} to do the calculation.  Denote by \(\rankdf\) the rank of \(\dd F|_{\pi_1(p)}\); we have \(\ld_i = 0\) for \(i>\rankdf\).  Let \(A\) be a \(2\rankdf\times\rankdf\) matrix whose columns form an orthonormal set.  Namely, \(A^* A = \bfI_\rankdf\).  Write
\begin{align*}
    A = \begin{bmatrix} U \\ V \end{bmatrix} ~\text{ for }U,V\in\BR^{\rankdf\times\rankdf} ~,~
    \text{ then } U^* U + V^* V = \bfI_\rankdf ~.
\end{align*}
Let \(u_j\) be the \(j^{\text{th}}\) row of \(U\), and \(v_j\) be the \(j^{\text{th}}\) row of \(V\).
For \(\CI = \{1\leq i_1<\cdots<i_\ell\leq\rankdf\}\), denote by \(A_\CI\) the \(\rankdf\times\rankdf\) matrix whose \(i^{\text{th}}\) row is \(v_j\) if \(j\in\CI\), and is \(u_j\) if \(j\notin\CI\).  It is clear that \(\det(A_\CI)\) is the evaluation of \(\Dt_\CI\) (see \eqref{form:special-basis}) on the column space of \(A\).  This immediately leads to the following lemma.

\begin{lem} \label{lem:comass-Stiefel}
    For \(\ld_1\geq\cdots\geq\ld_\rankdf\geq 0\), the comass of \eqref{form:Theta-tangent-normal} is the maximum of
    \begin{align*}
        \det(U) - \sum_{\ell=2}^\rankdf(-1)^\ell(\ell-1)\sum_{\CI:|\CI|=\ell}(\prod_{j\in \CI}\ld_j) \det(A_\CI)
    \end{align*}
    over row vectors \(u_1,\ldots,u_\rankdf,v_1,\ldots,v_\rankdf\) with \(\sum_{j=1}^\rankdf (u_j^* u_j + v_j^* v_j) = \bfI_\rankdf\).
\end{lem}

To elaborate, note that
\begin{align} \label{eqn:trace-r}
    \rankdf &= \tr(U^* U + V^* V) = \tr(U U^*) + \tr(V V^*) = \sum_{j=1}^\rankdf |u_j|^2 + \sum_{k=1}^\rankdf |v_k|^2 ~.
\end{align}
It is convenient to write \({\sum_{j=1}^\rankdf|u_j|^2}/{\rankdf}\) as \((\cos\ta)^2\), and \({\sum_{j=1}^\rankdf|v_j|^2}/{\rankdf}\) as \((\sin\ta)^2\), where \(\ta\in[0,\frac{\pi}{2}]\).

By the Hadamard inequality and the AM-GM inequality,
\begin{align} \label{eqn:det-U}
    \det(U) \leq \prod_{j=1}^\rankdf|u_j|^2 \leq \Bigl(\frac{\sum_{j=1}^\rankdf |u_j|^2}{\rankdf}\Bigr)^{\frac{\rankdf}{2}} = (\cos\ta)^\rankdf ~.
\end{align}
Together with the Cauchy--Schwarz inequality, for any \(1\leq\ell\leq\rankdf\),
\begin{align*}
    \sum_{\CI:|\CI|=\ell}\det(A_\CI) &\leq \sum_{\CI:|\CI|=\ell} \bigl(\prod_{j\notin\CI}|u_j|\bigr)\bigl(\prod_{k\in \CI}|v_k|\bigr) \leq \bigl(\sum_{\CI:|\CI|=\ell}\prod_{j\notin \CI}|u_j|^2\bigr)^{\frac{1}{2}} \bigl(\sum_{\CI:|\CI|=\ell}\prod_{k\in\CI}|v_k|^2\bigr)^{\frac{1}{2}} ~.
\end{align*}
With the help of the Maclaurin inequality,
\begin{align} \label{eqn:det-LA}
    \sum_{\CI:|\CI|=\ell}\det(A_\CI) &\leq \binom{\rankdf}{\ell}\Bigl(\frac{\sum_{j=1}^\rankdf |u_j|^2}{\rankdf}\Bigr)^{\frac{\rankdf-\ell}{2}}\Bigl(\frac{\sum_{k=1}^\rankdf |v_k|^2}{\rankdf}\Bigr)^{\frac{\ell}{2}}  = \binom{\rankdf}{\ell} (\cos\ta)^{\rankdf-\ell} (\sin\ta)^\ell ~.
\end{align}
Applying \eqref{eqn:det-U} and \eqref{eqn:det-LA} to Lemma~\ref{lem:comass-Stiefel} leads to the following proposition.

\begin{prop} \label{prop:comass-rough-bound}
    Let \(\rankdf = \rank(\dd F|_{\pi_1(p)})\), and let \(\ld_1\geq\cdots\geq\ld_\rankdf> 0\) be the singular values of \(\dd F|_{\pi_1(p)}\).  The comass of \eqref{form:Theta-tangent-normal} is no greater than
    \begin{align} \label{eqn:comass-rough-bound}
        \max \Bigl\{ (\cos\ta)^\rankdf + \sum_{\ell=2}^\rankdf \Ld_\ell (\ell-1) \binom{\rankdf}{\ell} (\cos\ta)^{\rankdf-\ell} (\sin\ta)^\ell : {\ta\in[0,\frac{\pi}{2}]} \Bigr\} ~,
    \end{align}
    where \(\Ld_\ell = \max\{\prod_{j\in\CI}\ld_j:|\CI|=\ell\}\).
\end{prop}


\begin{thm} \label{thm:2-dilation-crude}
Let $F:\Omega\subset \mathbb{R}^n\rightarrow \mathbb{R}^m$ be a smooth map whose
graph is a minimal submanifold of $\mathbb{R}^{n+m}$. Suppose that \( \rankdf = \sup_{\bx\in\Om}\bigl(\rank \dd F|_{\bx}\bigr) \geq 2 \). If the singular values $\{\lambda_i\}$ of \(\dd F\) satisfy
    \begin{align} \label{eqn:2-dilation-crude}
        \ld_j\ld_k \leq \frac{1}{(\rankdf-1)^2}
    \end{align}
    for any \(1\leq j<k\leq \rankdf\) and at every point of \(\Om\), then  \(\Ta(F)\) is a calibration.
\end{thm}
\begin{proof}
    If \(\ld_j\ld_k \leq \tau^2\) for any \(j\neq k\), then \(\Ld_\ell \leq \tau^\ell\) for any \(\ell\geq 2\). For any \(\tau > 0\), let
    \begin{align}
        \sff(\ta,\tau) &= (\cos\ta)^\rankdf + \sum_{\ell=2}^\rankdf \tau^\ell (\ell-1) \binom{\rankdf}{\ell} (\cos\ta)^{\rankdf-\ell} (\sin\ta)^\ell  \label{eqn:f-kappa-rankdf}
    \end{align}
    With the condition \eqref{eqn:2-dilation-crude}, it suffices to show that \(\sff(\ta,\frac{1}{\rankdf-1})\leq 1\) for \(0\leq\ta\leq\frac{\pi}{2}\).  We compute
    \begin{align*}
        \sff\bigl(\ta,\frac{1}{\rankdf-1}\bigr) &\leq (\cos\ta)^2 + \Bigl[\sum_{\ell=2}^\rankdf \frac{1}{(\rankdf-1)^\ell} (\ell-1) \binom{\rankdf}{\ell} \Bigr] (\sin\ta)^2 \\
        &= (\cos\ta)^2 + (\sin\ta)^2 = 1 ~.
    \end{align*}
    This completes the proof.
\end{proof}

The Taylor series expansion of \eqref{eqn:f-kappa-rankdf} at \(\ta = 0\) is \( 1 + \frac{1}{2}\rankdf \bigl( (\rankdf - 1)\tau^2 - 1 \bigr)\ta^2 + O(\ta^3)\).  This suggests that \(\ld_j\ld_k \leq \frac{1}{\rankdf - 1}\) is a necessary condition for \(\Ta(F)\) having comass one.  In the following theorem, we prove that this condition, up to replacing \(1\) by a smaller absolute constant, is sufficient.

\begin{thm} \label{thm:2-dilation-refined}
    There exists a constant \(\vep > 0\) with the following property.  For any minimal map \(F: \Om\subset\BR^n\to\BR^m\) with \( \rankdf = \sup_{\bx\in\Om}\bigl(\rank \dd F|_{\bx}\bigr) \geq 2 \), if the singular values of \(\dd F\) satisfy
    \begin{align} \label{eqn:2-dilation-refined}
        \ld_j\ld_k \leq \frac{\vep}{\rankdf-1}
    \end{align}
    for any \(1\leq j<k\leq \rankdf\) and at every point of \(\Om\), then \(\Ta(F)\) is a calibration form.
\end{thm}
\begin{proof}
The main task is to show that \( \sff(\ta,\sqrt{\frac{\vep}{\rankdf-1}}) \leq 1\).  Throughout the proof, we assume that \(0\leq\ta\leq\frac{\pi}{2}\).

{\textit{Step 1: the square of \(\sff(\ta,\tau)\)}.}
Rewrite \(\sff(\ta,\tau)\) as follows:
\begin{align*}
    \sff(\ta,\tau) &= 2(\cos\ta)^\rankdf + (\cos\ta)^\rankdf\sum_{\ell=0}^\rankdf (\ell-1)\binom{\rankdf}{\ell}(\tau\tan\ta)^\ell \\
    &= 2(\cos\ta)^\rankdf + (\cos\ta)^\rankdf \bigl((\rankdf-1)\tau\tan\ta-1\bigr)(1+(\tau\tan\ta))^{\rankdf-1} ~.
\end{align*}

Since
\begin{align*}
    &\quad \bigl((\rankdf-1)(\tau\tan\ta)-1\bigr)^2 (1+(\tau\tan\ta))^{2\rankdf-2} \\
    &= \bigl((\rankdf-1)^2(\tau\tan\ta)^2 - 2(\rankdf-1)(\tau\tan\ta) + 1\bigr)^2 \sum_{\ell=0}^{2\rankdf-2}\binom{2\rankdf-2}{\ell}(\tau\tan\ta)^\ell
\end{align*}
and
\begin{align*}
    (\rankdf-1)^2\binom{2\rankdf-2}{\ell-2} - 2(\rankdf-1)\binom{2\rankdf-2}{\ell-1} + \binom{2\rankdf-2}{\ell} &= (\ell-1)\binom{2\rankdf}{\ell}\bigl(\frac{\ell \rankdf}{2(2\rankdf-1)} - 1\bigr) ~,
\end{align*}
we have
\begin{align*}
    &\quad (\cos\ta)^{2\rankdf} \bigl((\rankdf-1)\tau\tan\ta-1\bigr)^2(1+(\tau\tan\ta))^{2\rankdf-2} \\
    &= \sum_{\ell = 0}^{2\rankdf} \tau^\ell \Bigl[ (\ell-1)\binom{2\rankdf}{\ell} \bigl(\frac{\ell\rankdf}{2(2\rankdf -1)} - 1\bigr) \Bigr](\cos\ta)^{2\rankdf-\ell}(\sin\ta)^\ell
\end{align*}
It follows that
\begin{align*}
    &\quad \bigl( \sff(\ta,\tau) \bigr)^2 \\
    &= 4(\cos\ta)^{2\rankdf} + 4\sum_{\ell=0}^\rankdf(\ell-1)\binom{\rankdf}{\ell}\tau^\ell(\cos\ta)^{2\rankdf-\ell}(\sin\ta)^\ell + \Bigl[ \sum_{\ell=0}^\rankdf(\ell-1)\binom{\rankdf}{\ell}\tau^\ell(\cos\ta)^{\rankdf-\ell}(\sin\ta)^\ell \Bigl]^2 \\
    &= (\cos\ta)^{2\rankdf} + \sum_{\ell = 2}^\rankdf \tau^\ell \Bigl[ (\ell-1)\binom{2\rankdf}{\ell} \bigl(\frac{\ell\rankdf}{2(2\rankdf -1)} - 1\bigr) + 4(\ell-1)\binom{\rankdf}{\ell} \Bigr](\cos\ta)^{2\rankdf-\ell}(\sin\ta)^\ell \\
    &\qquad + \sum_{\ell = \rankdf+1}^{2\rankdf} \tau^\ell \Bigl[ (\ell-1)\binom{2\rankdf}{\ell} \bigl(\frac{\ell\rankdf}{2(2\rankdf -1)} - 1\bigr) \Bigr](\cos\ta)^{2\rankdf-\ell}(\sin\ta)^\ell
\end{align*}

{\textit{Step 2: an upper bound for \((\sff(\ta,\tau))^2\)}.}
Since \(\ell\geq 2\), \(\binom{\rankdf}{\ell} \leq \frac{1}{2^\ell}\binom{2\rankdf}{\ell} \leq \frac{1}{4}\binom{2\rankdf}{\ell}\), and
\begin{align*}
    \binom{2\rankdf}{\ell} \bigl(\frac{\ell\rankdf}{2(2\rankdf -1)} - 1\bigr) + 4\binom{\rankdf}{\ell}
    &\leq \frac{\ell\rankdf}{2(2\rankdf -1)} \binom{2\rankdf}{\ell} \leq \frac{\ell}{3} \binom{2\rankdf}{\ell} ~.
\end{align*}
It follows that
\begin{align} \label{eqn:f-square-estimate1}
    \bigl( \sff(\ta,\tau) \bigr)^2 &\leq (\cos\ta)^{2\rankdf} + \sum_{\ell=2}^{2\rankdf} \tau^\ell \frac{(\ell-1)\ell}{3} \binom{2\rankdf}{\ell} (\cos\ta)^{2\rankdf-\ell}(\sin\ta)^\ell ~.
\end{align}
Consider the sum over odd \(\ell\)'s, and apply the Cauchy--Schwarz inequality: 
\begin{align*}
    &\quad \sum_{i=1}^{\rankdf-1} \tau^{2i} \frac{i(2i+1)}{3} \binom{2\rankdf}{2i+1} (\cos\ta)^{2\rankdf-2i-2}(\sin\ta)^{2i} \bigl(2(\cos\ta)(\tau\sin\ta)\bigr) \\
    &\leq \sum_{i=1}^{\rankdf-1} \tau^{2i} \frac{i(2i+1)}{3} \binom{2\rankdf}{2i+1} (\cos\ta)^{2\rankdf-2i-2}(\sin\ta)^{2i} \Bigl(\frac{2i+1}{2\rankdf-2i+1}(\cos\ta)^2 + \frac{2\rankdf-2i+1}{2i+1}\tau^2(\sin\ta)^2 \Bigr) \\
    &= \sum_{i=1}^{\rankdf-1} \tau^{2i} \frac{i(2i+1)}{3} \frac{2\rankdf-2i}{2i+1} \binom{2\rankdf}{2i}\frac{2i+1}{2\rankdf-2i+1} (\cos\ta)^{2\rankdf-2i}(\sin\ta)^{2i} \\
    &\quad + \sum_{j=1}^{\rankdf-1} \tau^{2j+2} \frac{j(2j+1)}{3} \frac{2j+2}{2\rankdf-2j-1} \binom{2\rankdf}{2(j+1)} \frac{2\rankdf-2j+1}{2j+1} (\cos\ta)^{2\rankdf-2j-2}(\sin\ta)^{2j+2} \\
    &= \sum_{i=1}^{\rankdf-1} \tau^{2i} \frac{i(2i+1)}{3} \Bigl[1 - \frac{1}{2\rankdf-2i+1}\Bigr] \binom{2\rankdf}{2i} (\cos\ta)^{2\rankdf-2i}(\sin\ta)^{2i} \\
    &\quad + \sum_{i=2}^{\rankdf} \tau^{2i} \frac{i(2i-2)}{3} \Bigl[1 + \frac{2}{2\rankdf-2i+1}\Bigr] \binom{2\rankdf}{2i} (\cos\ta)^{2\rankdf-2i}(\sin\ta)^{2i} \\
    &\leq \sum_{i=1}^{\rankdf-1} \tau^{2i} i^2 \binom{2\rankdf}{2i} (\cos\ta)^{2\rankdf-2i}(\sin\ta)^{2i} + \sum_{i=2}^{\rankdf} \tau^{2i} 2i^2 \binom{2\rankdf}{2i} (\cos\ta)^{2\rankdf-2i}(\sin\ta)^{2i} \\
    &\leq \sum_{j=1}^{\rankdf} \tau^{2j} 3j^2 \binom{2\rankdf}{2j} (\cos\ta)^{2\rankdf-2j}(\sin\ta)^{2j} ~.
\end{align*}

Together with \eqref{eqn:f-square-estimate1},
\begin{align}
    \bigl( \sff(\ta,\tau) \bigr)^2 &\leq (\cos\ta)^{2\rankdf} + \sum_{j=1}^{\rankdf} \tau^{2j} \frac{(2j-1)2j}{3} \binom{2\rankdf}{2j} (\cos\ta)^{2\rankdf-2j}(\sin\ta)^{2j} \notag \\
    &\quad + \sum_{i=1}^{\rankdf-1} \tau^{2i+1} \frac{2i(2i+1)}{3} \binom{2\rankdf}{2i+1} (\cos\ta)^{2\rankdf-2i-1}(\sin\ta)^{2i+1} \notag \\
    &\leq (\cos\ta)^{2\rankdf} + \sum_{j=1}^{\rankdf} \tau^{2j} \Bigl[\frac{(2j-1)2j}{3} + 3j^2\Bigr] \binom{2\rankdf}{2j} (\cos\ta)^{2\rankdf-2j}(\sin\ta)^{2j} ~. \label{eqn:f-square-estimate2}
\end{align}

{\textit{Step 3: bounding \((\sff(\ta,\tau))^2\)} by \(1\).}
Since
\begin{align*}
    1 &= (\cos^2\ta + \sin^2\ta)^\rankdf = (\cos\ta)^{2\rankdf} + \sum_{j=1}^{\rankdf} \binom{\rankdf}{j}(\cos\ta)^{2\rankdf-2j}(\sin\ta)^{2j} ~,
\end{align*}
it suffices to show that there exists an \(\vep > 0\) such that
\begin{align*}
    \frac{\vep^j}{(\rankdf-1)^j}\frac{13}{3}j^2\binom{2\rankdf}{2j} &\leq \binom{\rankdf}{j} ~.
\end{align*}
We compute
\begin{align*}
    \frac{\vep^j}{(\rankdf-1)^j}\frac{13}{3}j^2\binom{2\rankdf}{2j} \Bigl[\binom{\rankdf}{j}\Bigr]^{-1}
    &= \frac{\vep^j}{(\rankdf-1)^j} \frac{13}{3} j^2 \prod_{k=1}^j \frac{2\rankdf-2j+2k-1}{2k-1} \\
    &= \vep^j \frac{13}{3} \Bigl(\prod_{k=1}^j\frac{2\rankdf-2j+2k-1}{p-1}\Bigr) \frac{j^2}{\prod_{\ell=1}^j (2\ell-1)} \\
    &\leq \vep^j \frac{13}{3} \bigl(3\cdot 2^{j-1}\bigr)\cdot2 ~,
\end{align*}
and it is easy to see that we can choose sufficiently small \(\vep>0\) such that the above expression is no greater than \(1\).
\end{proof}

\section{Lipschitz Solution to the Minimal Graph System} \label{sec:Lipschitz}

Note that the above discussions work perfectly well for \(F = (f_1,\cdots,f_m)\) being \(C^2\).  In this section, we study \(\Ta(F)\) for \emph{locally Lipschitz} \(F\).  We say that \emph{\(F\) satisfies the minimal graph system \eqref{eqn:minimal-standard} weakly} if for \(\af = 1,\ldots,m\),
\begin{align} \label{eqn:weakly-minimal}
    \int_{\Om}\sum_{j,k=1}^n \sqrt{g}(g^{-1})^{jk} (\pl_k f_\af) (\pl_j\vph)\,\dd x_1\cdots\dd x_n &= 0
\end{align}
for any smooth \(\vph:\Om\to\BR\) of compact support.  Geometrically, \eqref{eqn:weakly-minimal} means that \(\Sm\), the graph of \(F\), is a critical point of the volume functional with respect to (compactly supported) \emph{outer variations}, i.e.\ variations in the \(F\)-direction (equivalently, in the \(\BR^m\)-directions).  When \(F\) is \(C^2\), \eqref{eqn:weakly-minimal} is equivalent to the vanishing of the mean curvature vector, and hence \(\Sm\) is critical with respect to \emph{any} variations.  In other words, when \(F\) is \(C^2\), \eqref{eqn:weakly-minimal} is equivalent to
\begin{align} \label{eqn:stationary}
    \left\{
        \begin{aligned}
            \sum_{k=1}^{n} \pl_k\bigl(\sqrt{g}\,(g^{-1})^{jk}\bigr) &=0
            \qquad \text{for } j=1,\ldots,n ~, \\
            \sum_{j,k=1}^{n} \pl_j\bigl(\sqrt{g}\,(g^{-1})^{jk}\,(\pl_k f_\af)\bigr) &= 0
            \qquad \text{for } \af=1,\ldots,m ~.
        \end{aligned}
    \right.
\end{align}
The first line of \eqref{eqn:stationary} corresponds to the criticality of the graph with respect to \emph{inner variations} of the volume functional.  However, when \(F\) is only locally Lipschitz, it is not known whether \eqref{eqn:weakly-minimal} implies \eqref{eqn:stationary} (weakly).  Note that satisfying \eqref{eqn:stationary} weakly is equivalent to \(\Sm\) being stationary (see \cite{Simon-83}*{Section 16}).  A conjecture of Lawson and Osserman \cite{Lawson-Osserman-77}*{Conjecture 2.1} asserts that for locally Lipschitz maps, being outer critical and being stationary are equivalent.

Now, suppose that \(F\) is locally Lipschitz and satisfies the minimal graph system weakly.
It follows that the right-hand side of \eqref{form:Theta-g} is a weakly closed differential form (with the coefficient functions in \(L^\infty_{\text{loc}}(\Om\times\BR^m)\)).  In other words, the right-hand side of \eqref{form:Theta-g} defines a locally real flat cochain on \(\Om\times\BR^m\); see \cite{Federer-74}*{4.6} and \cite{Federer-69}*{4.1.19}.  It is convenient to abuse notation and denote this locally real flat cochain by \(\Ta(F)\).

Denote by \(T_F\) the locally\footnote{\(T_F\) needs not have finite mass.} integral current associated with the graph of \(F\).  If the comass of the right-hand side of \eqref{form:Theta-g} is \(1\) almost everywhere on \(\Om\times\BR^m\), it follows from \cite{Federer-69}*{\({\mathbf{F}}(\af)\) on p.377 and \({\mathbf{M}}(\phi)\) on p.358} that \(T_F\) is a locally mass minimizing current.  More precisely, let \(W\subset\Om\times\BR^m\) be an open set with \(\ol{W}\subset\Om\times\BR^m\), and let \(S\) be an \((n+1)\)-integral current in \(\Om\times\BR^m\) with \(\spt S\subset W\).  Since \(\Ta(F)\) is weakly closed, \(\Ta(F)(\pl S) = 0\), and hence
\begin{align*}
    \mathbf{M}(T_F \mathbin{\llcorner} W) &= \Ta(F)(T_F \mathbin{\llcorner} W) = \Ta(F)(T_F \mathbin{\llcorner} W + \pl S) \leq \mathbf{M}(T_F \mathbin{\llcorner} W + \pl S) ~.
\end{align*}

This yields current-theoretic versions of Theorems~\ref{thm:2-dilation-crude} and \ref{thm:2-dilation-refined}.
\begin{prop} \label{prop:Lipschitz-mass-minimizing}
    Let \(F: \Om\subset\BR^n\to\BR^m\) be a locally Lipschitz weak solution to the minimal graph system.  Denote \(\esssup_{\bx\in\Om}\bigl(\rank \dd F|_{\bx}\bigr)\) by \(\rankdf\).   Suppose one of the following holds:
    \begin{enumerate}
        \item \(\rankdf \leq 1\);
        \item \(\rankdf \geq 2\), and the singular values of \(\dd F\) satisfy \(\ld_j\ld_k \leq \frac{1}{(\rankdf-1)^2}\) a.e.;
        \item \(\rankdf \geq 2\), and the singular values of \(\dd F\) satisfy \(\ld_j\ld_k \leq \frac{\vep}{\rankdf-1}\) a.e., where \(\vep>0\) is given by Theorem~\ref{thm:2-dilation-refined}.
    \end{enumerate}
    Then, the current associated with the graph of \(F\) is locally mass minimizing.
\end{prop}

Proposition~\ref{prop:Lipschitz-mass-minimizing} implies the following corollary, which confirms the Lawson--Osserman conjecture \cite{Lawson-Osserman-77}*{Conjecture 2.1} when the map satisfies the \(2\)-dilation condition. .

\begin{cor} \label{LO}
    For a locally Lipschitz weak solution \(F: \Om\subset\BR^n\to\BR^m\) to the minimal graph system that satisfies the \(2\)-dilation conditions of Proposition~\ref{prop:Lipschitz-mass-minimizing}, the graph of $F$ is stationary. Moreover, \(F\) is smooth.
\end{cor}

\begin{proof}
    The first assertion follows from the fact that a locally mass-minimizing current is stationary (see for instance \cite{Simon-83}*{Lemma 33.2}).
    
    For the smoothness of \(F\), suppose that \(\Sm\) has a singular point \(p\).  Since \(F\) is locally Lipschitz, any tangent cone of \(\Sm\) at \(p\) is a minimal graph of an entire, Lipschitz function, satisfying the same \(2\)-dilation condition.  According to the Bernstein theorem\footnote{The proof of \cite{Wang-03}*{Theorem A} uses blow-down and Federer's dimension reduction, and works for the minimal graph of an entire, Lipschitz function.} (\cite{Wang-03}*{Theorem A} and \cite{Jing-Yang-2021}*{Theorem 1.3}), the tangent cone is an affine \(n\)-plane.  By the Allard Regularity Theorem \cite{Allard-72}, \(p\) is a smooth point of \(\Sm\).
\end{proof}


\end{document}